\documentclass{article}
\usepackage{theorem,graphicx,amssymb,amsmath}
\usepackage[usenames,dvipsnames]{xcolor}
\usepackage[utf8]{inputenc}
\usepackage{hyperref}
\usepackage[maxbibnames=99]{biblatex}
\theorembodyfont{\slshape}

\newtheorem{theorem}{Theorem}

\newtheorem{observation}{Observation}

\def\QED{\ensuremath{{\square}}}

\def\markatright#1{\leavevmode\unskip\nobreak\quad\hspace*{\fill}{#1}}
\newenvironment{proof}
  {\begin{trivlist}\item[\hskip\labelsep{\bf Proof.}]}
  {\markatright{\QED}\end{trivlist}}

\title{Non-crossing Monotone Paths and Binary Trees in Edge-ordered  Complete Geometric Graphs}

\author{
  Frank Duque \thanks{Instituto de Matem\'aticas, Universidad de Antioquia, Medell\'in, Colombia. \texttt{rodrigo.duque@udea.edu.co}}
  \footnote{Partially supported by CONACYT (Mexico), grant 253261.} \and
  Ruy Fabila-Monroy \thanks{ Departamento de Matem\'aticas, CINVESTAV. \texttt{ruyfabila@math.cinvestav.edu.mx}} \footnotemark[2] \and
  Carlos Hidalgo-Toscano \thanks{INFOTEC Centro de Investigación e Innovación en Tecnologías de la Información y Comunicación. \texttt{carlos.hidalgo@infotec.mx}} \footnotemark[2] \ \and
  Pablo P\'erez-Lantero \thanks{Departamento de Matemática y Ciencia de la Computación, Universidad de Santiago, Chile. \texttt{pablo.perez.l@usach.cl}. Partially supported by projects DICYT 041933PL Vicerrector\'ia de Investigaci\'on, Desarrollo e Innovaci\'on USACH (Chile), and Programa Regional STICAMSUD 19-STIC-02.
}
  }

\begin{document}
\maketitle

\begin{abstract}
An edge-ordered graph is a graph with a total ordering of its edges.
A path $P=v_1v_2\ldots v_k$ in an edge-ordered graph is called increasing if $(v_iv_{i+1}) > (v_{i+1}v_{i+2})$ for all $i = 1,\ldots,k-2$;
it is called decreasing if $(v_iv_{i+1}) < (v_{i+1}v_{i+2})$ for all $i = 1,\ldots,k-2$. We say that $P$ is monotone if it is 
increasing or decreasing. A rooted tree $T$ in an edge-ordered graph is called monotone if either every path from the root of to a leaf 
is increasing or every path from the root to a leaf is decreasing.

Let $G$ be a graph. In a straight-line drawing $D$ of $G$, its vertices are drawn as different points in the plane and its edges are straight line segments. 
Let $\overline{\alpha}(G)$ be the maximum integer such that every edge-ordered straight-line drawing of $G$ 
contains a monotone non-crossing path of 
length $\overline{\alpha}(G)$. 
Let $\overline{\tau}(G)$ be the maximum integer such that every edge-ordered straight-line drawing of $G$ 
contains a monotone non-crossing complete binary tree of size $\overline{\tau}(G)$. 
In this paper we show that $\overline \alpha(K_n) = \Omega(\log\log n)$, $\overline \alpha(K_n) = O(\log n)$, $\overline \tau(K_n) = \Omega(\log\log \log n)$ and $\overline \tau(K_n) = O(\sqrt{n \log n})$.
\end{abstract}

\section{Introduction}
An \emph{edge-ordering} of a graph is a total order of its edges, an \emph{edge-ordered} graph is a graph with an edge-ordering. A path $P=v_1v_2\ldots v_k$ in an edge-ordered graph is called increasing if $(v_iv_{i+1}) > (v_{i+1}v_{i+2})$ for all $i = 1,\ldots,k-2$; it is called decreasing if $(v_iv_{i+1}) < (v_{i+1}v_{i+2})$ for all $i = 1,\ldots,k-2$. We say that $P$ is \emph{monotone} if it is increasing or decreasing. 
Let $G$ be a graph. Let $\alpha(G)$ be the maximum integer such that $G$ has a monotone path of length $\alpha(G)$ under any edge-ordering. The parameter $\alpha(G)$ is often called the \emph{altitude} of $G$.

Chv\'atal and Koml\'os~\cite{ChvatalKomlos} were the first to pose, in 1971, the problem of estimating the altitude of $K_n$. This problem has turned out to be challenging and has been studied by several researchers. The first related result is due to Graham and Kleitman~\cite{IncreasingPaths}, who showed in 1973 the following bounds:
\[ \frac{1}{2} \left( \sqrt{4n-3}-1 \right) \le \alpha(K_n) < \frac{3}{4}n.\]
They also conjectured the upper bound to be right order of magnitude of $\alpha(K_n)$.
More than ten years after, Calderbank, Chung, and Sturtevant~\cite{BlockSums} improved the upper bound to
\[ \alpha(K_n) \le \left(\frac{1}{2}+o(1)\right)n. \]
The lower bound by Graham and Kleitman remained the best known for over 40 years, until Milans~\cite{milans2017monotone} showed in 2017 that
\[ \left(\frac{1}{20}-o(1) \right) \left(\frac{n}{\lg n} \right)^{2/3} \le \alpha(K_n).\]
Recently, Buci\'c et al.~\cite{bucic2018nearly} showed that
\[ \alpha(K_n) \ge  \frac{n}{2^{O(\sqrt{\log n \log \log n})}} = n^{1-o(1)},\]
almost closing the gap between the upper and lower bounds.

In the meantime, several variations and specific cases of the altitude of a graph have been studied.
Bialostocki and Roditty~\cite{AMonotone} showed in 1987 that a graph $G$ has altitude at least three if and only if $G$ contains as a subgraph an odd cycle of length at least five or one of six fixed graphs. 

Yuster \cite{BoundedDegree} studied in 2001 the parameter $\alpha_{\Delta}$, defined as the maximum of $\alpha(G)$ over all graphs with maximum degree at most $\Delta$. He proved that $\Delta(1-o(1)) \le \alpha_\Delta \le \Delta +1$.
That same year Roditty, Shoham, and Yuster \cite{SparseGraphs} gave bounds on the altitude of several families of sparse graphs. In particular, if $G$ is a planar graph then $\alpha(G) \le 9$ and there exist planar graphs with $6 \le \alpha(G)$. They also proved that if $G$ is a bipartite planar graph then $\alpha(G) \le 6$ and there exist bipartite planar graphs with $4 \le \alpha(G)$.

Mynhardt, Burger, Clark, Falvai  and Henderson~\cite{girth} characterized in 2005 cubic graphs with girth at least five and altitude three. They also showed that if $G$ is an $r$-regular graph ($r \ge 4$) and has girth at least five, then $\alpha(G) \ge 4$.

De Silva, Molla, Pfender, Retter and Tait~\cite{hypercube} showed in 2015  that $d/\log d \le \alpha(Q_d)$, where $Q_d$ is the $d$-dimensional hypercube.

Lavrov and Loh~\cite{RandHam} studied in 2016  the length of a longest monotone path in  $K_n$ under a random edge-ordering. They showed that, with probability at least $1/e-o(1)$, $K_n$ contains a monotone Hamiltonian path. They also conjectured that, given a random ordering,  $K_n$ contains a monotone Hamiltonian path with probability tending to 1. Shortly after, Martinsson \cite{randomconjecture} proved this conjecture.


In this paper we study $\alpha(G)$ and similar parameters in the context of straight-line drawings of graphs. In a drawing
$D$ of $G$ in the plane, the vertices of $G$ are drawn as different points, and
every edge is drawn as a straight-line segment connecting the two points representing
its vertices. We require that in $D$ the edges do not pass through a point
representing a vertex other than their endpoints.

Let $D$ be a straight-line drawing of a graph $G$. Let $\overline \alpha(D)$ be the minimum over all edge-orderings of $D$ of the maximum length of a non-crossing monotone path in $D$. We denote by $\overline \alpha(G)$ the minimum of $\overline \alpha(D)$ over all straight-line drawings of $G$. 

Now we define parameters similar to $\overline \alpha$ related to binary trees instead of paths. From now on, all binary trees are complete and rooted. 

A rooted tree in an edge-ordered graph $G$ is increasing (decreasing) if every path from the root to a leaf is increasing (decreasing). It is called monotone if it is increasing or decreasing. Let $\overline \tau_+(D)$ ($\overline \tau_-(D)$) be the minimum over all edge-orderings of the maximum size of a non-crossing increasing (decreasing) binary tree in $D$. Let $\overline\tau_+(G)$ ($\overline\tau_-(G)$) be the minimum of $\overline \tau_+(D)$ ($\overline \tau_-(D)$) over all straight-line drawings of $G$. 

Similarly, let $\overline \tau(D)$ be the minimum over all edge-orderings of the maximum size of a non-crossing monotone binary tree in $D$ and $\overline\tau(G)$ the minimum of $\overline \tau(D)$ over all straight-line drawings of $G$.

In this paper we prove that $\overline{\alpha}(K_n)=\Omega(\log\log n)$ and $\overline{\alpha}(K_n)=O(\log n)$. For the parameter $\overline \tau$ we prove that $\overline \tau(K_n) = \Omega(\log \log \log n)$ and $\overline \tau(K_n) = O(\sqrt{n \log n})$. As an intermediate result, if we are interested in bounding the size of increasing or decreasing binary trees but not both, we prove that $\overline \tau_+(K_n) = O(\log n)$ and $\overline \tau_-(K_n) = O(\log n)$.

\subsection{Monotone Non-crossing Paths} \label{convex}

In this section we give bounds for $\overline \alpha (K_n)$. A straight-line drawing $D$ of a graph $G$ is said to be \emph{convex} if its vertices are in convex position. 


\begin{observation} \label{twoedges}
	Let $S$ be a set of points in convex position and $\ell$ a straight line that partitions $S$ into two nonempty sets $U$ and $V$. The maximum length of a non-crossing polygonal chain whose vertices alternate between $U$ and $V$ and whose edges increase in slope is two.
\end{observation}
\begin{proof}
	Assume that a polygonal chain $P$ of length three with the conditions of the statement exists. Let $p, q, r, s$ denote the vertices of $P$ and assume without loss of generality that $q$ has smaller $x$ coordinate than $r$. The edges of $P$ have increasing slope, thus both $p$ and $s$ lie to the right of the directed line from $q$ to $r$. Since the points $p, q, r, s$ are in convex position and the points $q$ and $s$ lie on a different semiplane than the points $p$ and $r$ with respect to $\ell$, the edges $(pq)$ and $(rs)$ are the diagonals of a convex quadrilateral. These diagonals intersect; a contradiction the non-crossing property of $P$.
\end{proof}

\begin{figure}[h] 
	\centering 
	\includegraphics[page=1]{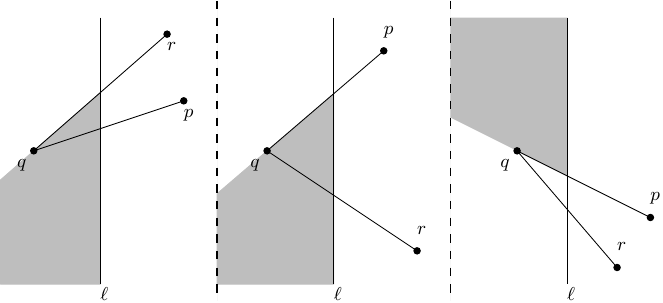}
	\caption{The fourth vertex of a polygonal with increasing slope must be in the shaded region.}
	\label{fig:twoedges}
\end{figure}

\begin{theorem}\label{thm:alphaLower}
	$\overline \alpha(K_n) = \Omega(\log\log n)$.
\end{theorem}
\begin{proof}	
	Let $D$ be an edge-ordered straight-line drawing of $K_n$ and assume without loss of generality that no two vertices of $D$ have the same $x$-coordinate. Let $v_1, \ldots, v_n$ be the vertices of $D$ ordered by increasing  $x$-coordinate, and let $H$ be the complete 3-uniform hypergraph with same vertex set as $D$. For $0\le i<j<k\le n$, color the edge $(v_i,v_j,v_k)$ of $H$ blue if $(v_iv_j) < (v_jv_k)$, otherwise color it red. From Ramsey's Theorem, there exists a complete monochromatic sub-hypergraph $K$ of size $m=\Omega(\log\log n)$ in $H$ with vertices $v_{i_1},v_{i_2},\ldots, v_{i_m}$. Then $P=v_{i_1},v_{i_2},\ldots, v_{i_m}$ is a monotone path of length $\Omega(\log\log n)$ in $D$. Note that $P$ is $x$-monotone, thus it has no crossings.
\end{proof}

\begin{theorem}\label{thm:alphaUpper}
	$\overline \alpha(K_n) = O(\log n)$.
\end{theorem}
\begin{proof}
	Let $D$ be a convex straight-line drawing of $K_n$. Let $\ell$ be a vertical line that partitions the vertices of $D$ into two no empty sets $U$ and $W$, each of size at most $\lceil n/2 \rceil$. Order the edges between $U$ and $W$ so that $e<e'$ if and only if the slope of $e$ is less than the slope of $e'$. Recursively order the edges of $D[U]$ and $D[W]$ as before. Extend these orders to $D$ by declaring the edges in $D[U] \cup D[W]$ to be less than those between $U$ and $W$. Let $P$ be a monotone path of maximum length in $D$. By Observation~\ref{twoedges}, there are at most two edges of $P$ between sets in the same level of recursion. Moreover, $P$ cannot have edges both in $D[U]$ and in $D[W]$, since there would be a subpath with an edge in $D[U]$, an edge between $U$ and $W$ and an edge in $D[W]$; this path cannot be monotone. Thus, the length $T(n)$ of $P$ satisfies the recursion $T(n) \le 2+T(n/2)$, which implies that $T(n) = O(\log n)$.
\end{proof}

\subsection{Monotone Non-crossing Binary Trees}

We begin this section with a lower bound for $\overline \tau( D)$, where $D$ is a convex straight-line drawing of $K_n$. We use the same argument used to bound $\overline \alpha(K_n)$.

\begin{theorem} \label{treesconvexlow}
	Let $D$ be an edge-ordered convex straight-line drawing of $K_n$. Then $\overline \tau(D) = \Omega(\log\log n)$.
\end{theorem}
\begin{proof}
	Let $v_1, v_2, \ldots, v_n$ be the vertices of $D$ ordered by increasing $x$-coordinate and let $H$ be the complete 3-uniform hypergraph with same vertex set as $D$. For $0\le i<j<k\le n$, color the edge $(v_i,v_j,v_k)$ of $H$ blue if $(v_iv_j) < (v_jv_k)$, otherwise color it red. From Ramsey's theorem, there exists a complete monochromatic sub-hypergraph $K$ of size $m=\Omega(\log\log n)$ in $H$. Assume without loss of generality that at least half of the vertices of $K$ belong to the lower convex hull of the vertices of $D$. Let $v_{i_1},\ldots,v_{i_m}$ denote those vertices ordered by increasing $x$-coordinate. Let $m'$ be the largest integer of the form $2^{h}-1$, for some integer $h$, such that  $m'\le m$. We embed a binary tree $T$ with vertices $v_{i_1},\ldots,v_{i_{m'}}$ as follows. Place the root of $T$ at $v_{i_1}$, and inductively place its left and right subtrees at the vertices $v_{i_2},\ldots,v_{i_{(m'+1)/2}}$ and $v_{i_{(m'+1)/2+1}},\ldots,v_{i_{m'}}$ with roots at $v_{i_2}$ and $v_{i_{(m'+1)/2+1}}$, respectively (see Figure \ref{fig:convexTree}). Note that $T$ is monotone and has no crossings by construction.
\end{proof}

\begin{figure}[h]
	\centering
	\includegraphics[page=2]{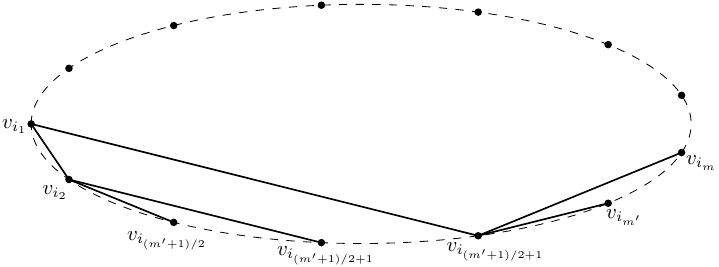}
	\caption{A non crossing binary tree in $H$.}
	\label{fig:convexTree}
\end{figure}

\begin{theorem} \label{theo:lowertau}
	$\overline \tau(K_n) = \Omega(\log \log \log n)$.
\end{theorem}
\begin{proof}
	Let $D$ be a straight-line drawing of $K_n$. By the Erd\H{o}s-Szekeres Theorem, there exists a set of $\Omega(\log n)$ vertices of $D$ in convex position; the bound follows from Theorem \ref{treesconvexlow}.
\end{proof}

When we search for monotone paths, there is no need to distinguish between  increasing and  decreasing paths---traversing an increasing path in opposite direction gives us a decreasing path of the same length and vice versa. This is not the case with binary trees. Theorem~\ref{theo:lowertau} guarantees that we can find a monotone binary tree of size $\Omega(\log\log\log n)$ in any edge-ordered straight-line drawing of $K_n$, which can be increasing or decreasing. 

An upper bound on $\overline \tau(K_n)$ must take into account both  increasing and  decreasing binary trees. We start by bounding $\overline  \tau_+(K_n)$ and $\overline \tau_-(K_n)$. 

\begin{theorem} \label{treesconvexup1}
	$\overline \tau_+(K_n) = O(\log n)$ and $\overline \tau_-(K_n) = O(\log n)$.
\end{theorem}
\begin{proof}
	Let $D$ be a convex straight-line drawing of $K_n$. We give an edge-ordering of $D$ such that largest non-crossing increasing binary tree has size $O(\log n)$. The proof that $\tau_-(D) = O(\log n)$ is analogous. 
	
	The supporting line of every edge $e=(uv)$ of $D$ partitions the vertices of $D \setminus \{u,v\}$ into two sets, let $S_e$ be the smaller one. Construct an edge-ordering of $D$ such that $e < e'$ if $|S_e| < |S_{e'}|$. Let $T$ be a largest non-crossing increasing binary tree in $D$ and denote  its root by $r$.
	
	Let $L$ and $R$ be the convex hulls of the left and right subtrees of $T$, respectively. We claim that $L$ and $R$ are disjoint. Suppose this is not the case. Order the vertices of $K_n$ counterclockwise starting with $r$ and let $l', v_1, \ldots, v_s, r'$ be the vertices of $L \cup R$ as they appear in this order. Let $e$ be the left edge of $r$. No edge of $T$ can have an endpoint in $S_e$: any such edge $f$ must intersect $e$ or have both endpoints in $S_e$, in which case $f$ belongs to the left subtree of $T$; this is impossible since $|S_f|<|S_e|$. Thus, $l'$ is the root of the left subtree of $T$ and by an analogous argument $r'$ is the root of the right subtree of $T$. 	
	There must be a vertex of $R$ followed by a vertex of $L$ in this order, otherwise $L$ and $R$ are disjoint; let $v_i$ be the first such vertex of $R$ in this order. In the path that joins $v_i$ with $r'$ in $T$, there exists at least one edge $v_jv_k$ with $j \le i < k$. The vertices $l'$ and $v_{i+1}$ lie on different sides of the supporting line of $v_jv_k$, so the path that joins them in $T$ must intersect $v_jv_k$. This is a contradiction, since $T$ is non-crossing.
	
	Let $\ell$ be a straight line through $r$ that separates $L$ and $R$. The line $\ell$ partitions the vertices of $D \setminus r$ in two parts, one with less than $n/2$ vertices. Let $S$ be this part and let $T'$ be the subtree of $T$ contained in $S$. 
	
	Let $h'$ be the height of $T'$. Note that for every vertex $v$ of $T'$ with children $u$ and $w$ either the subtree $T_u$ rooted at $u$ is contained in $S_{(vw)}$ or the subtree $T_w$ rooted at $w$ is contained in $S_{(vu)}$. Assume without loss of generality that the first case happens. 
	
	\begin{figure}[h]
		\centering
		\includegraphics[page=3]{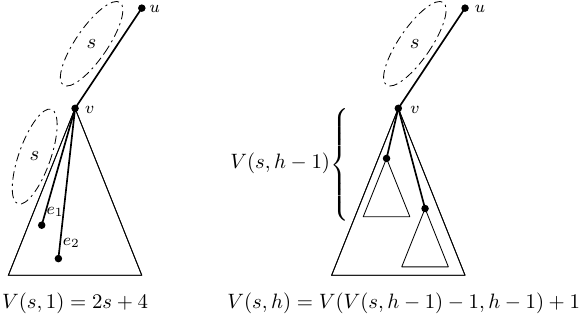}
		\caption{$T_{s, 1}$ and $T_{s, h}$.}
		\label{fig:trees}
	\end{figure}
	
	Let $T_{s,h}$ denote an increasing tree with root $u$ such that: 
	\begin{itemize}
		\item The vertex $u$ has only one child $v$ and $|S_{(uv)}| = s$
		\item The vertex $v$ is the root of a binary tree with height $h$
	\end{itemize}
	Let $V(s, h)$ denote the minimum number of vertices of $S$ needed to embed $T_{s,h}$. If $h=1$, the complete binary tree rooted at $v$ consists of two edges, $e_1, e_2$; furthermore, $|S_{(uv)}| = s$ implies that $|S_{e_1}|\ge s$ and $|S_{e_2}|\ge s+1$. Thus, $V(s, 1) = 2s+4$. (See figure \ref{fig:trees}). Note that if $h > 1$, we need at least $V(s, h-1)$ vertices to embed the left subtree of $v$, which implies that we need $V(V(s, h-1)-1, h-1)$ vertices to embed the right subtree of $v$. Thus we obtain the following recurrence: 
	
	\[V(s, h) \ge V(V(s, h-1)-1, h-1)+s+1.\]
	
	We show by induction on $h$ that $V(s, h) \ge (s+1)2^{2^{h-1}}$. This certainly holds for the base case, assume it also holds for $h-1$. We have that:
	
	\begin{eqnarray*}
		V(s, h) &\ge&  V(V(s, h-1)-1, h-1)    \\
		&\ge& V(s, h-1)\cdot 2^{2^{h-2}} \\
		&\ge& (s+1)\cdot 2^{2^{h-2}} \cdot 2^{2^{h-2}} \\
		&\ge& (s+1)2^{2^{h-1}}
	\end{eqnarray*}
	
	Note that $h' \ge V(0, h') = \Omega(2^{2^{h'-1}})$, which implies that $h' = O(\log \log n)$. The theorem follows.	
	%
\end{proof}

\begin{figure}[h]
	\centering
	\includegraphics[page=4]{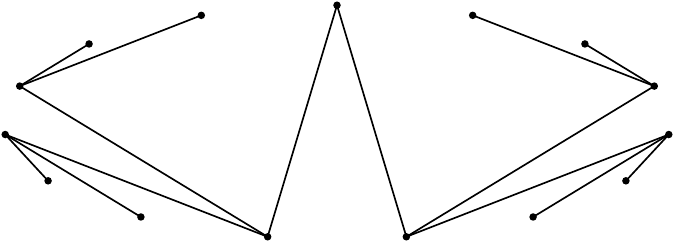}
	\caption{A decreasing binary tree of linear size under the edge-ordering used in Theorem \ref{treesconvexup1}.}
	\label{fig:badtree}
\end{figure}

The edge-ordering used in the proof of Theorem~\ref{treesconvexup1} forbids large increasing binary trees, but it is possible to find a decreasing binary tree of linear size (see Figure \ref{fig:badtree}). The proof of Theorem~\ref{treesconvexup2} gives an edge-ordering that forbids both increasing and decreasing large binary trees.

\begin{theorem} \label{treesconvexup2}
	$\overline \tau(K_n) = O(\sqrt{n \log n})$.
\end{theorem}
\begin{proof}
	Let $D$ be a convex straight-line drawing of $K_n$. Let $v_1, \ldots, v_n$ denote the vertices of $D$ in counterclockwise order. Let $m=\left \lceil \sqrt{n/\log n} \right \rceil$ and partition the vertices into groups $S_1, S_2, \ldots, S_m$ of consecutive vertices such that each one has size at most $\left \lceil \sqrt{n \log n} \right \rceil$. Order the edges that have endpoints within $S_i$ using the same edge-ordering as in Theorem \ref{treesconvexup1} so that the largest non-crossing decreasing binary tree contained in $S_i$ has size $O(\log n)$. We refer to those edges as red edges and to the edges that have endpoints in different groups as blue edges. Order the blue edges by increasing slope. Furthermore, order the edges in such a way that every blue edge is greater than every red edge.

	\begin{figure}[h]
		\centering
		\includegraphics[page=5]{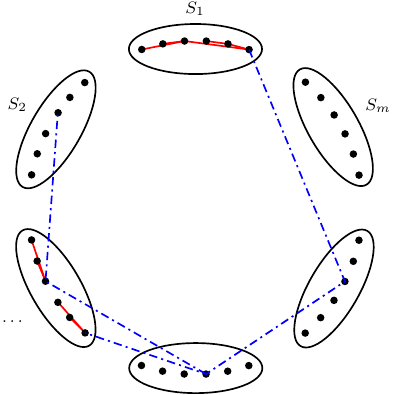}
		\caption{A  decreasing  binary tree with respect to the edge-ordering of Theorem~\ref{treesconvexup2}.}
		\label{fig:T6_dec}
	\end{figure}
	
	Let $T$ be a decreasing binary tree with respect to this edge-ordering and let $r$ be its root. Note that $T$ consists of a possibly empty blue binary tree $T_b$ and a forest of red complete binary trees such that the roots of the red trees are leaves of $T_b$ (Figure \ref{fig:T6_dec}). We claim that the subgraph $T_{i,j}$ of $T_b$ induced by blue the edges between two different groups $S_i$ and $S_j$ has at most one connected component. Suppose for the sake of a contradiction that this does not happen. Choose two connected components such that $r$ is in at most one of them, let $e$ be and edge from the first component and $f$ an edge from the second component. Suppose that one of the support lines (say, the support line of $e$) leaves $f$ and $r$ in different semiplanes. Then, since the vertices of $T_b$ are in convex position, the path from any endpoint of $f$ to $r$ must intersect $e$. Therefore, $r$ lies between the support lines of $e$ and $f$. Without loss of generality, $r$ is not in the connected component where $e$ belongs. Any path from $r$ to an endpoint of $e$ must go first through a vertex in some $S_k$, where $k \ne i, j$. This path intersects $e$ or $f$, a contradiction.
	
	Since $T_{i,j}$ has at most one connected component (which by Observation~\ref{twoedges} is a tree of height at most two), its number of edges is at most a constant. Let $u_1, \ldots, u_m$ be vertices on a circle ordered counterclockwise. Add an edge between $u_i$ and $u_j$ if and only if there is an edge in $T_b$ between $S_i$ and $S_j$. Since $T_b$ is non-crossing, the resulting graph is also non-crossing and has $O(\sqrt{n/\log n})$ edges, which implies that $T_b$ has $O(\sqrt{n/\log n} )$ edges. There are $O(\sqrt{n/\log n})$ leaves in $T_b$, thus there are $O(\sqrt{n/\log n})$ red trees and by Theorem~\ref{treesconvexup1} each has size $O(\log n)$. Thus, $T$ has size $O(\sqrt{n \log n})$.
	
	\begin{figure}[h]	
		\centering
		\includegraphics[page=6]{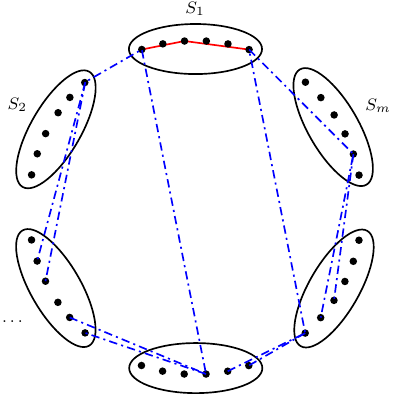}
		\caption{An increasing complete binary tree with respect to the edge-ordering of Theorem~\ref{treesconvexup2}.}
		\label{fig:T6_inc}
	\end{figure}

	Now consider an increasing complete binary tree $T$. Note that $T$ consists of a possibly empty red binary tree $T_r$ and a forest of blue binary trees such that the roots of the blue trees are leaves of $T_r$ (Figure \ref{fig:T6_inc}). The tree $T_r$ is contained in some $S_i$, thus it has size $O(\sqrt{n \log n})$. Identify all the roots of the blue trees with a vertex of $S_i$--this produces a non crossing blue tree. There are  $O(\sqrt{n \log n})$ blue edges with an endpoint in $S_i$. We can bound the number of remaining blue edges as before by $O(\sqrt{n / \log n})$. Therefore, $T$ has size $O(\sqrt{n \log n})$.
\end{proof}

\small
\printbibliography

\end{document}